\documentclass[11pt,leqno]{amsart}

\usepackage{amsmath}
\usepackage{amssymb}
\usepackage{a4wide}
\usepackage{bbm}
\usepackage{graphics}
\usepackage{epsfig}

\newcommand{\Subsection}[1]{\subsection{ #1} ${}^{}$}

\newtheorem{theorem}{Theorem}[section]
\newtheorem{lemma}[theorem]{Lemma}
\newtheorem{proposition}[theorem]{Proposition}
\newtheorem{definition}[theorem]{Definition}
\newtheorem{remark}[theorem]{Remark}
\newtheorem{corollary}[theorem]{Corollary}

\newcounter{hypo}

\newenvironment{hyp}{
 \begin{enumerate}
\setcounter{enumi}{\value{hypo}} \item}{\stepcounter{hypo} \end{enumerate}}

\makeatletter
 \@addtoreset{equation}{section}
 \makeatother
 
\title[SSF for perturbed periodic Schr\"odinger operators]
{Spectral shift function for slowly varying perturbation of periodic
Schr\"odinger operators.}

\author[M. Dimassi]{Mouez Dimassi}
\address{Mouez Dimassi, LAGA, (UMR CNRS 7539), Univ. Paris 13, F-93430 Villetaneuse, France}
\email{dimassi@math.univ-paris13.fr}
\author[M. Zerzeri]{Maher Zerzeri}
\address{Maher Zerzeri, LAGA, (UMR CNRS 7539), Univ. Paris 13, F-93430 Villetaneuse, France}
\email{zerzeri@math.univ-paris13.fr}


\begin{document}

\begin{abstract} In this paper we study the asymptotic expansion
of the spectral shift function for the slowly
varying perturbations of periodic Schr\"odinger operators.
We give a weak and pointwise asymptotic expansions
in powers of $h$ of the derivative of the spectral shift function
corresponding to the pair $\big(P(h)=P_0+\varphi(hx),P_0=-\Delta+V(x)\big),$
where $\varphi(x)\in {\mathcal C}^\infty(\mathbb R^n,\mathbb R)$
is a decreasing function, ${\mathcal O}(|x|^{-\delta})$ for some $\delta>n$
and $h$ is a small positive parameter.
Here the potential $V$ is real, smooth and periodic with respect to
a lattice $\Gamma$ in ${\mathbb R}^n$.
To prove the pointwise asymptotic expansion of the spectral shift function, we establish
a limiting absorption Theorem for $P(h)$.
\end{abstract}

\maketitle

\setcounter{tocdepth}{2}
\tableofcontents

\vfill\break
\section{Introduction}\label{intro}

The aim of this paper is to give an asymptotic expansion of the
spectral shift function for the slowly varying perturbations of
periodic Schr\"odinger operator:
\begin{equation}\label{one}
P(h)=P_0+\varphi(hx),\quad h>0,
\end{equation}
$$
P_0=-\Delta_x+V(x),
$$
Here $V$ is a real-valued, ${\mathcal C}^\infty$ function and
periodic with respect to a lattice $\Gamma$ of $\mathbb
R^n$.

The Hamiltonian $P(h)$ describes the quantum motion of an
electron in a crystal placed in an external field. There are many
works devoted to the spectral properties of this model, see
\cite{AlDeHe89_01, BiYa95_01, Bu87_01, Di98_01,  Di02_01, Di05_01, DiZe03_01, Ge90_01,
GeMaSj91_01, GeNi98_01, GuRaTr88_01, HoSpTe01_01, Sl49_01}.

We assume  that $\varphi\in {\mathcal C}^\infty({\mathbb R}^n; {\mathbb
R})$ and satisfies the following estimate:
for all $\alpha\in \mathbb{N}^n,$ there exists $C_\alpha>0$ such that
\begin{equation}\label{as1}
\vert \partial^\alpha_x\varphi(x)\vert\leq C_\alpha(1+\vert x\vert)^{-\delta-|\alpha|},
\quad \forall x\in \mathbb R^n,\quad
\text{with} \  \boxed{{\delta>n}}.
\end{equation}

The operators $P_0, P(h)$ are self-adjoint on $H^2({\mathbb R}^n)$.
Under the assumption (\ref{as1}) we show in Theorem \ref{w-SSFh} below
that the operator $\big[f(P(h))-f(P_0)\big]$ belongs to the trace class for all
$f\in {\mathcal C}^\infty_0(\mathbb R)$. Following the general setup we
define the spectral shift function, SSF,
$\xi_h(\mu):=\xi(\mu;P(h),P_0)$ related to the pair $(P(h),P_0)$ by
\begin{equation}
{\rm tr}\big[f(P(h))-f(P_0)\big]=-\langle \xi_h'(\cdot),
f(\cdot)\rangle=\int_{\mathbb R} \xi_h(\mu) f'(\mu)
d\mu,\quad \forall f\in {\mathcal C}^\infty_0(\mathbb R).
\end{equation}
By this formula $\xi_h$ is defined modulo a constant but for the
analysis of the derivative $\xi'_h(\mu)$ this is not
important. See \cite{Kr53_01} and \cite{BiYa93_01}.

The SSF may be considered as a generalization of the eigenvalues
counting function. The notion of SSF was first singled out by the
outstanding theoretical physicist I-M.Lifshits in his investigations
in the solid state theory, in 1952, see \cite{Li52_01}. It was
brought into mathematical use in M-G. Kre\u{\i}n's famous paper
\cite{Kr53_01}, where the precise statement of the problem was given
and explicit representation of the SSF in term of the perturbation
determinant was obtained. The work of M-G. Kre\u{\i}n's on the SSF
has been described in detail in \cite{BiYa93_01}. For more details
about the interpretation of SSF we refer to the survey by D. Robert
\cite{Ro99_01} and to the monograph by D-R. Yafaev
\cite[Chapter 8]{Ya00_01}.

In the case where $V=0$, the asymptotic behavior of the SSF of the
Schr\"odinger operator has been intensively studied in different
aspects (see \cite{Co81_01,Gu85_01,MaRa78_01,PePo82_01,Ro92_01,Ro94_01,RoTa87_01}
and the references given there). In the semi-classical regime (i.e. $H(h)=-h^2\Delta_x+\varphi(x),
(h\searrow 0))$ the Weyl type asymptotics of
$\xi_h(\cdot)$ with sharp remainder
estimate has been obtained (see \cite{Ro92_01,Ro94_01,RoTa87_01,RoTa88_01}).
On the other hand, if an energy
$\mu>0$ is non-trapping for the classical hamiltonian
$p(x,\zeta)=\vert \zeta\vert^2+\varphi(x)$ (i.e. for all $(x,\zeta)\in
p^{-1}\{\mu\}$, $\vert {\rm exp}(tH_p)(x,\zeta)\vert \rightarrow
\infty$ when $t\rightarrow \infty$) a complete asymptotic
expansion in powers of $h$  of $\xi'_h(\mu)$ has been
obtained (see \cite{Ro92_01,Ro94_01,RoTa87_01,RoTa88_01}).
Similar results are well-known for the SSF at
high energy (see \cite{Bu71_01,Co81_01,PePo82_01,Po82_01,Ro91_01}).

There are only few works treating the SSF in perturbed periodic
Schr\"odinger operator. See \cite{BiYa95_01}, \cite{Di05_01} and
also \cite{GeNi98_01}. In \cite{Di05_01} the connection between
the resonances of $P(h)$ and the SSF associated to the pair
$\big(P(h), P_0\big)$ were studied.
Under the assumption that $\varphi$ is analytic in
some conic complex neighborhood of the real axis and that $P(h)$ has
no resonances in a small complex neighborhood of some interval $I$
the first author obtained a full asymptotic expansion in powers of $h$ of
the derivative of SSF:
\begin{equation}\label{Dimasym}
\xi'(\mu;h)\sim \sum_{j=0}^{\infty} b_j(\mu) h^{j-n},\quad
h\searrow 0,
\end{equation}
uniformly with respect to $\mu\in I$.

Nevertheless, there is a lot of examples of perturbed periodic Schr\"odinger operator
that the perturbation $\varphi$ does not satisfies the analyticity assumption.

In this paper, we improve the result of \cite{Di05_01} concerning
the behavior of the derivative of SSF by removing the analyticity
assumption on the potential $\varphi$. Our proof is based on a
limiting absorption principle and some arguments due to D. Robert
and H. Tamura, see \cite{RoTa84_01} and \cite{RoTa88_01}.
For $V\not=0,$ the limiting absorption theorem is new
(see Theorem \ref{abslimit1}).
They are in harmony with the physical
intuition which argues that, when $h$ sufficiently small, the main
effect of the periodic potential $V$ consists in changing the
dispersion relation from the free kinetic $|k|^2$ to the modified
kinetic energy $\lambda_p(k)$ given by the $p$th band.

By the method of effective hamiltonian spectral problems of $P(h)$
can be reduced to similar problem of systems of
$h-$pseudodifferential operators (See \cite{Di02_01} and also
\cite{DiZe03_01}). Using a well-known results on
$h-$pseudodifferential calculus we get the asymptotic
\eqref{Dimasym} in the sense of distributions. If the values of the
principal term of the effective hamiltonian are contained in
non-trapping energy region we prove a limiting absorption
principle for $P(h)$ (Theorem \ref{abslimit1}) and we get a pointwise
asymptotic expansion for the derivative of the spectral shift function.

{\sl The paper is organized as follows}:
In the next section, we recall some well-known results concerning the spectra of a periodic
Schr\"odinger operator (Subsection \ref{premi}) and we state the
assumptions and the results precisely (Subsection \ref{main}).
We give an outline of the proofs in Subsection
\ref{outproof}. Section \ref{proofs} is devoted to the proofs.
Roughly, we introduce a class
of symbols and the corresponding $h$-Weyl operators (Subsection \ref{pseudocalculus}).
In the subsection \ref{EffHam} we recall the effective Hamiltonian method and
we give a representation of the derivative of the spectral shift function, denoted by
$\zeta'_h(\cdot).$
The proof of the weak asymptotic expansion of $\xi'_h$ is given
in Subsection \ref{pw-SSFh}.
We establish a limiting absorption principle for $P(h)$ in the subsection \ref{abslimTh}
At last, the pointwise asymptotic expansion of $\xi'_h$
is proved in Subsection \ref{p-dSSFh}.

\section{Statements}

\Subsection{Preliminaries}\label{premi}

Let $\Gamma=\underset{i=1}{\overset{n}{\oplus}}{\mathbb Z}e_i$ be a lattice generated
by some basis $(e_1,e_2,\cdots,e_n)$ of ${\mathbb R}^n.$
The dual lattice $\Gamma^*$ is given by
$\Gamma^*:=\{\gamma^*\in {\mathbb R}^n;\
\langle \gamma | \gamma^* \rangle \in 2\pi{\mathbb Z},\
\forall \gamma\in \Gamma\}.$
A fundamental domain of $\Gamma$ (resp. $\Gamma^*$)
is denoted by $E$ (resp. $E^*$). If we identify opposite edges
of $E$ (resp. $E^*$) then it becomes a flat torus
denoted by $\displaystyle {\mathbb T}={\mathbb R}^n/\Gamma$ (resp.
$\displaystyle {\mathbb T}^*={\mathbb R}^n/\Gamma^*).$

Let $V$ be a real-valued potential, ${\mathcal C}^\infty$ and
$\Gamma-$periodic. For $k\in {\mathbb R}^n,$ we define the
operator $P(k)$ on $L^ 2({\mathbb T})$ by
$P(k):=(D_{y}+k)^2+V(y).$ The operator $P(k)$ is a
semi-bounded self-adjoint with $k$-independent domain
$H^ 2({\mathbb T}).$ Since the resolvent of $P(k)$ is compact,
$P(k)$ has a complete set of (normalized) eigenfunctions
$\Phi_{n}(\cdot,k)\in H^{2}({\mathbb T}),\ n\in {\mathbb N},$
called Bloch functions. The corresponding eigenvalues accumulate at
infinity and we enumerate them according to their multiplicities,
$\lambda_{1}(k)\leq \lambda_{2}(k)\leq \cdots .$ The operator
$P(k)$ satisfies the identity $e^{-iy\cdot
\gamma^*}P(k)e^{iy\cdot \gamma^*}=P(k+\gamma^*),\ \forall
\gamma^*\in \Gamma^*,$ then for every $p\geq 1,$ the function
$k\mapsto \lambda_{p}(k)$ is $\Gamma^{*}-$periodic.

Ordinary perturbation theory shows that $\lambda_{p}(k)$ are
continuous functions of $k$ for any fixed $p,$ and
$\lambda_{p}(k)$ is even an analytic function of $k$ near any
point $k_{0}\in  {\mathbb T}^{*}$ where $\lambda_{p}(k_{0})$ is
a simple eigenvalue of $P(k_{0}).$ The function $\lambda_{p}(k)$
is called the band function and the closed intervals
$\Lambda_{p}:=\lambda_{p}({\mathbb T}^*)$ are called bands. See
\cite{ReSi78_01}, \cite{Sj91_01} and also \cite{Sk85_01, Sk85_02}.

Consider the self-adjoint operator on $L^{2}({\mathbb R}^n)$
with domain $H^{2}({\mathbb R}^n)$:
\begin{equation}\label{freeoperator}
P_{0}=-\Delta_x +V(x),\quad \text{where}\ \Delta_x=\sum_{j=1}^n\frac{\partial^2}{\partial x_j^2}.
\end{equation}

The spectrum of $P_0$ is absolutely continuous
(see \cite{Th73_01}) and consists of the bands $\Lambda_{p},\ p=1,2,\cdots$.
Indeed,
$\displaystyle
\sigma(P_{0})=\sigma_{\rm ac}(P_{0})=\underset{p\geq
1}{\cup}\Lambda_{p}.$ See also \cite{Sh79_01}.

\begin{definition}\label{Fermisurface}
Let $\mu\in {\mathbb R}$ and $F(\mu)=\big\{k\in {\mathbb
T}^*;\ \mu\in \sigma\big(P(k)\big)\big\}$ the corresponding
Fermi-surface.
\begin{itemize}
\item[a)] We will say that $\mu\in \sigma(P_{0})$ is a simple energy level
if and only if $\mu$ is a simple eigenvalue of $P(k),$ for
every  $k\in F(\mu).$

\item[b)] Assume that $\mu$ is a simple energy level of $P_{0}$
and let $\lambda(k)$ be the unique eigenvalue defined on a
neighborhood of $F(\mu)$ such that $\lambda(k)=\mu,$ 
for all $k\in F(\mu).$ We say that $\mu$ is a
non-critical energy  of $P_0$ if $d_{k}\lambda(k)\not=0$ for
all $k\in F(\mu).$
\end{itemize}
\end{definition}

Note that in one dimension case $F(\mu)$ is just a finite set of points.

Now, let us recall some well-known facts about the density of states
associated with $P_0,$ see \cite{Sh79_01}.
The density of states measure $\rho$
is defined as follows:
\begin{equation}
\rho(\mu):=\frac{1}{(2\pi)^n}\sum_{p\geq 1}
\int_{\{k\in E^*;\ \lambda_p(k)\leq \mu\}}\, dk.
\end{equation}

Since the spectrum of $P_{0}$ is absolutely continuous,
the measure $\rho$ is absolutely continuous with respect to
the Lebesgue measure $d\mu.$ Therefore
the density of states, $\frac{d\rho}{dE}(E),$ of $P_{0}$
is locally integrable.

\Subsection{Results}\label{main}

We now consider the perturbed periodic Schr\"odinger operator:
\begin{equation}
P(h):=P_0+\varphi(hx),\quad h\searrow 0,
\end{equation}
where $\varphi\in C^\infty({\mathbb R}^n;\mathbb R)$ and satisfies:
\begin{hyp}\label{h1}
There exists $\delta>0$ such that $\forall\alpha\in \mathbb N^n,$
$\exists C_\alpha>0$ s.t.
$$
\big|\partial_x^\alpha\varphi(x)\big|\leq C_\alpha(1+\vert
x\vert)^{-\delta-|\alpha|}\ \text{uniformly on}\ x\in \mathbb R^n.
$$
\end{hyp}
The operator $P(h)$ is self-adjoint, semi-bounded on $L^2(\mathbb
R^n)$ with domain $H^2(\mathbb R^n).$
\\
The assumption \ref{h1} and the perturbation theory (Weyl theorem) give:
\begin{equation}
\sigma_{\rm ess}\big(P(h)\big)=\sigma_{\rm
ess}(P_0)=\sigma(P_0)=\bigcup_{p\geq 1}\Lambda_p.
\end{equation}
Recall that $\sigma_{\rm ess}(A),$ the essential spectrum of $A,$ is
defined by $\sigma_{\rm ess}(A)=\sigma(A)\setminus \sigma_{\rm
disc}(A),$ where $\sigma_{\rm disc}(A)$  is the set of isolated
eigenvalues of $A$ with finite multiplicity. Here $A$ is an
unbounded operator on a Hilbert space.

Our first theorem in this section concerns the weak asymptotic of $\zeta'_h(\mu).$
Let $I=]a,b[\subset \mathbb R.$
\begin{theorem}[Weak asymptotic]\label{w-SSFh}
Assume \ref{h1} with \fbox{$\delta>n$}.
For $f\in {\mathcal C}_0^\infty(I),$ the operator $\Big[f(P(h))-f(P_{0})\Big]$
is of trace class and
\begin{equation}
{\rm tr}\Big[f(P(h))-f(P_{0})\Big]\sim
h^{-n}\sum_{j=0}^{+\infty}a_{j}(f)h^j,\quad \text{when}\
h\searrow 0,
\end{equation}
with
\begin{equation}
a_{0}(f)=(2\pi)^{-n}\sum_{p\geq 1}\int_{{\mathbb R}^n_x}\int_{E^*}
\Big[f\big(\lambda_p(k)+\varphi(x)\big)-
f\big(\lambda_p(k)\big)\Big]\,dk\,dx.
\end{equation}
The coefficients $f\rightarrow a_j(f)$ are distributions of finite
order $\leq j+1.$  Moreover, if $\mu$ is a non-critical energy
of $P_0$ for all $\mu\in I,$ then $a_j(f)=-\langle
\gamma_j(\cdot),f\rangle,$ for all $f\in {\mathcal C}_0^\infty(I).$ Here
$\gamma_j(\mu)$ are smooth functions of $\mu\in I$. In
particular,
\begin{equation}\label{a1}
\gamma_0(\mu)=\frac{\rm d}{{\rm d}\mu}\left[\int_{\mathbb
R^n_x}\Big\{\rho\big(\mu\big)-\rho\big(\mu-\varphi(x)\big)\Big\} \,
dx\right].
\end{equation}
\end{theorem}
The proof of Theorem \ref{w-SSFh} is contained in subsection \ref{pw-SSFh}.

Let $[a,b]\subset \mathbb R.$ Assume that:
\begin{hyp}\label{h2}
for all $\mu\in [a,b],$ $\mu$ is a non-critical energy of
$P_0$.
\end{hyp}
For all $\mu\in [a,b],$ let $\lambda(k)$ be the unique
eigenvalue defined on a neighborhood of $F(\mu)$ such that
$\lambda(k)=\mu.$ We assume that for all $(k,r)\in \mathbb{T}^*\times{\mathbb R}^n$
such that $\mu=\lambda(k)+\varphi(r)\in \sigma (P_0)\cap [a,b]$, $\mu$ is a simple energy level,
and that:
\begin{hyp}\label{h3}
$\vert\nabla\lambda(k)\vert^2-r\nabla\varphi(r)\Delta\lambda(k)>0,$
for all $(k,r)$ s.t. $\lambda(k)+\varphi(r)\in [a,b].$
\end{hyp}
{\bf Remark.} Note that the assumption \ref{h2} is fulfilled in the bottom of the spectrum
of $P_0$. Moreover, assuming \ref{h2} the hypothesis \ref{h3} is satisfied
if $\Vert\varphi\Vert_\infty+\Vert x\nabla\varphi\Vert_\infty<<1$,
(see \cite{ReSi78_01}, \cite{Sk85_01, Sk85_02}).

Our main result concerning the derivative of the
spectral shift function is the following.
\begin{theorem}[Pointwise asymptotic]\label{dSSFh}
Assume \ref{h2}, \ref{h3} and \ref{h1} with \fbox{$\delta>n$}.
Then the following asymptotic expansion holds:
\begin{equation}\label{DimZerasym1}
\zeta'_h(\mu)\sim h^{-n}\sum_{j\geq
0}\gamma_j(\mu)h^{j}\quad {\rm as}\ h\searrow 0,
\end{equation}
uniformly for $\mu\in [a,b].$ \\
The coefficients $\big(\gamma_j(\mu)\big)_{j\geq 0}$ are given in Theorem \ref{w-SSFh}.
Furthermore, this expansion has derivate in $\mu$ to any order.
\end{theorem}

\begin{remark} Theorems \ref{w-SSFh} and \ref{dSSFh} still true also in the case
when the potential $\varphi(x)$ dependent on $h$, i.e.
$\varphi(x,h)=\varphi(x)+h\varphi_1(x)+h^2\varphi_2(x)\cdots$ in
$S^{\delta}(1).$ See the next section for the definition of
$S^{\delta}(1).$ (Subsection \ref{pseudocalculus}).
\end{remark}

\Subsection{Outline of the proofs}\label{outproof} Let
$Q_1(h)=Q^w_1(x,hD_x)$, $Q_0(h)=Q^w_0(x,hD_x)$ be two
$h$-pseudodifferential operators such that $Q_j(h)=Q_j^*(h)$,
($j=0,1$), and
$$
(Q_1(h)+i)^{-1}-(Q_0(h)+i)^{-1}
$$
is an operator of trace class. In this case, Theorem \ref{w-SSFh} is
well-known see \cite{Ro87_01} and the references therein. On the
other hand, in the case of non-trapping geometrie, the asymptotic
follows from the results of Robert-Tamura (see \cite{RoTa84_01,
RoTa87_01, RoTa88_01} and also \cite{Ro94_01}). The main ingredient
in the Robert-Tamura method is the limiting absorption theorem and
the construction of a long-time parametrix for the time-dependent
equation
$$
\big(hD_t-Q^w_j(x,hD_x)\big)U_j(t)=0, \quad U_j(0)=I.
$$
In our case $P(h)=-\Delta+V(x)+\varphi(hx)$ is not an
$h$-pseudodifferential operator. In fact, when $(h\searrow 0)$,
there are two spatial scales in equation \eqref{one}. The first one
of the order of the linear dimension $\gamma$ of the periodicity
cell and the second one  of order $\frac{\gamma}{h}$ on which the
perturbation of the potential varies appreciably. To remedy this we
reduce the study of $P(h),$ to the one of a system of
$h$-pseudodifferential operators. More precisely, following
\cite{Di93_01, Di02_01, DiZe03_01, GeMaSj91_01, GuRaTr88_01}, we can
reduce the spectral study of $P(h)$ near any fixed energy $z$ to the
study of a finite system of $h$-pseudodifferential operators,
$E_{-+}(z,h)$, acting on $L^2(\mathbb T^{*n};\mathbb C^N)$. In
general, for the reduced problem, the dependence on the spectral
parameter is non-linear. However, in the case of simple band (see
assumption \ref{h2}) we show that
$$
E_{-+}(z,h)=z-\big(\lambda(k)+\varphi(r)+hK_1(k,r)+h^2K_2(k,r;z,h)\big),
$$
where $K_1\in S^{\delta+1}(\mathbb T^*\times {\mathbb R}^n)$ and
$K_2(\cdot;z,h)\in S^{\delta+2}({\mathbb T}^*\times {\mathbb R}^n),$
holomorphic with respect to $z$ in a small complex neighborhood
$\Omega$ of a bounded interval $I$. (See \eqref{cal}). Now,
considering $s\in \Omega\cap \mathbb R$ as a parameter and assuming
\ref{h3}, we can apply the Robert-Tamura approach to the hamiltonian
$B_s(k,-hD_k;h):=\lambda(k)+\varphi(-hD_k)+hK_1^w(k,-hD_k)+h^2G^w(k,-hD_k;s,h)$
where $G$ satisfies the same properties as $K_2$, and we obtain the
theorem \ref{dSSFh}. Here we use the following crucial argument: the
assumption \ref{h3} implies that the interval $I$ is a non-trapping
region of the classical hamiltonian associated to $B_s$ for all $s$
in the compact set $\Omega\cap \mathbb R$. In fact, $r\cdot
K_i(k,r)\in S^{\delta}({\mathbb T}^*\times {\mathbb R}^n)\subset
S^0({\mathbb T}^*\times {\mathbb R}^n),$ ($i=1,2$) then the
corresponding operators are bounded uniformly for $s\in \Omega\cap
\mathbb R$ and moreover, the principal symbol of $B_s$ does not
depend on $s$.

\section{Proofs}\label{proofs}

\Subsection{Definitions and notations}\label{pseudocalculus}

Let $H$ be a Hilbert space. The scalar product in $H$ will be denoted by
$\langle\cdot,\cdot\rangle$. The set of linear bounded operators
from $H_1$ to $H_2$ is denoted
by ${\mathcal L}(H_1,H_2)$.\\
For $(m,N)\in \mathbb R\times\mathbb N$ we denote by $S^m({\mathbb
T}^*\times\mathbb R^n; {\mathcal M}_N(\mathbb C))$ the space  of
$P\in C^\infty\big(\mathbb R^{2n}_{k,r}; {\mathcal M}_N(\mathbb
C)\big)$, $\Gamma^*$-periodic with respect to $k$, such that for all
$\alpha$ and $\beta$ in $\mathbb N^n$ there exists
$C_{\alpha,\beta}>0$ such that
\begin{equation}\label{symbolestimate}
\Vert \partial_r^\alpha\partial_k^\beta P(k,r)\Vert_{{\mathcal M}_N(\mathbb C)}
\leq C_{\alpha,\beta}\langle r\rangle ^{-m-|\alpha|},\quad \langle r\rangle=
\big(1+\vert r\vert^2\big)^{\frac{1}{2}},
\end{equation}
where ${\mathcal M}_N(\mathbb C)$ is the set of $N\times N$-matrices.
In the special case when $N=1$ (i.e., $P$ is real valued),
we will write $S^m({\mathbb
T}^*\times\mathbb R^n)$ instead of $S^m({\mathbb
T}^*\times\mathbb R^n; {\mathcal M}_1(\mathbb C))$.

If $P$ depends on a semi-classical parameter $h\in \rbrack
0,h_0\rbrack$ and possibly on other parameters as well, we require
\eqref{symbolestimate} to hold uniformly with respect to these
parameters. For $h$ dependent symbols, we say that $P(k,r;h)$ has an
asymptotic expansion in powers of $h$, and we write
$$
P(k,r;h)\sim\sum_{j=0}^\infty P_j(k,r) h^j,
$$
if for every $N\in \mathbb N$, $\displaystyle
h^{-(N+1)}\Big(P-\sum_{j=0}^N P_jh^j\Big)\in S^m\big({\mathbb
T}^*\times\mathbb R^n; {\mathcal M}_N(\mathbb C)\big)$.

For $P\in S^m({\mathbb T^*}\times\mathbb R^n; {\mathcal
M}_N(\mathbb C)) $, the $h$-Weyl operator $P=P^w(k,hD_k;h)={\rm
Op}_h^w(P)$ is defined by:
$$
P^w(k,hD_k;h)u(k)=(2\pi h)^{-n}\int\int e^{\frac{i}{h}(k-y)r}
P(\frac{k+y}{2},r;h)u(y)\,dy\,dr.\quad \text{Here}\
D_k=\frac{1}{i}\frac{\partial}{\partial k}.
$$

\Subsection{Effective Hamiltonian} \label{EffHam}

In this subsection, we recall the effective Hamiltonian method. More
precisely, we will construct a suitable auxiliary (so-called
Grushin) problem associated with the operator $\big(P(h)-z\big)$ for
$z$ in a small complex neighborhood of $I,$ where $I=[a,b]\subset
\mathbb R$ is some bounded interval. The reader can find more
details and the proofs of the results of this subsection in
\cite{HeSj90_01} (see also \cite{DiSj99_01,DiZe03_01,GeMaSj91_01}). 
For the reader convenience, let us point out the main
change in our situation and fix the notations.

Denote by $T_{\Gamma}$ the distribution in ${\mathcal S}'(\mathbb R^{2n})$ defined by
$\displaystyle
T_{\Gamma}(x,y)=\frac{1}{\text{vol}(E)h^{n}}
\sum_{\beta^*\in \Gamma^*}e^{i(x-hy)\frac{\beta^*}{h}}.
$
We recall that $E$ is a fundamental domain of $\Gamma.$

For $m\in \mathbb N,$ put $\displaystyle{\mathbb
L}^m:=\{u(x)T_{\Gamma}(x,y);\
\partial^{\alpha}_xu\in L^2(\mathbb R^n),\
\forall\alpha,\,|\alpha|\leq m\}.$

It was shown in \cite[Chapter 13, Proposition 13.5]{DiSj99_01}, that the operator $P(h)$ acting on
$L^2(\mathbb R^n)$ with domain $H^2(\mathbb R^n)$ is unitary equivalent to
\begin{equation}
{\mathbb P}_1(h):=\big(D_y+hD_x\big)^2+V(y)+\varphi(x),
\end{equation}
acting on ${\mathbb L}^0$ with domain ${\mathbb L}^2,$ and the following proposition holds.
\begin{proposition}\label{Grushinproblem}
Assume \ref{h1}. There exist $N \in {\mathbb N}$, a complex
neighborhood $\Omega$ of $I$, and a bounded operator $R_+$ in
${\mathcal L}\big({\mathbb L}^0;L^2(\mathbb T^*;\mathbb C^N)\big)$
such that for all $z\in \Omega$ and $0<h<h_0$ small enough,
the operator
\begin{equation}
{\mathcal P}_1(z,h):=
\begin{pmatrix}
{\mathbb P}_1(h) -z  & R_+^* \\
 R_+ & 0
\end{pmatrix}
:{\mathbb L}^2 \times L^2(\mathbb T^*;\mathbb C^N)
 \rightarrow {\mathbb L}^0 \times
L^2(\mathbb T^*;\mathbb C^N),
\end{equation}
is bijective with bounded two-sided inverse
\begin{equation}
{\mathcal E}_1(z,h):=
\begin{pmatrix}
E_1(z,h) &  E_{1,+}(z,h) \\
 E_{1,-}(z,h) & E_{1,-+}(z,h)
\end{pmatrix}.
\end{equation}
Here $E_{1,-+}:=E^{w}_{1,-+}(k,-hD_k;z,h)$ is an
$h-$pseudodifferential operator with symbol
\begin{equation}
E_{1,-+}(k,r;z,h)\sim \sum_{l\ge 0} E_{1,-+}^{l}(k,r;z)\, h^l,\quad
\forall\, 0<h<h_0,
\end{equation}
in $S^0\Big(\mathbb T^*\times \mathbb R^n;{\mathcal L}
(\mathbb C^N,\mathbb C^N)\Big)$.
\end{proposition}
\begin{remark}\label{freeGrushin}
\begin{itemize}
\item[(1)] We denote by ${\mathcal P}_0(z,h)$ and
$\displaystyle {\mathcal E}_0(z,h):=
\begin{pmatrix}
 E_0(z,h) & E_{0,+}(z,h) \\
E_{0,-}(z,h) & E_{0,-+}(z,h)
\end{pmatrix}
$
the operators given by Proposition \ref{Grushinproblem} when $\varphi=0.$
 \item[(2)] Note that,  $R_+$ depends only on the non-perturbed
periodic Schr\"odinger operator $P_0.$
See \cite[Proposition 2.1]{GeMaSj91_01} and
\cite[Chapter 13]{DiSj99_01}.
Therefore, we may take the same $R_+$ for
${\mathcal P}_1(z,h)$ and ${\mathcal P}_0(z,h).$
\end{itemize}
\end{remark}

The following well-known formulas are a consequence of
Proposition \ref{Grushinproblem} (see also \cite{HeSj90_01}),
for $j=0, 1.$
\begin{equation}\label{f1}
 \big({\mathbb P}_j(h)-z\big)^{-1}=
E_j(z,h)- E_{j,+}(z,h)
E_{j,-+}(z,h)^{-1} E_{j,-}(z,h),
\end{equation}
\begin{equation}\label{f2}
E_{j,-+}(z,h)^{-1}=-R_+\big({\mathbb P}_j(h)-z\big)^{-1} R_+^*\, ,
\end{equation}
and
\begin{equation}\label{f3}
\partial_z E_{j,-+}(z,h)=E_{j,-}(z,h) E_{j,+}(z,h).
\end{equation}
Here $\displaystyle{\mathbb P}_0(h):=\big(D_y+hD_x\big)^2+V(y).$

We observe that ${\mathcal P}_j(z,h)^{*}={\mathcal P}_j(\overline{z},h)$,
which implies that ${\mathcal E}_j(z,h)^{*}={\mathcal E}_j(\overline{z},h).$
From this, we deduce the following identity:
\begin{equation}\label{adjcc}
E_{j,-+}(z,h)^{*}=E_{j,-+}(\overline{z},h),\quad j=0,1.
\end{equation}

In the following, we write $[a_j]_{j=0}^1=a_1-a_0.$
\begin{lemma}\label{formula}
We have
\begin{equation}\label{f4}
\Big[E_{j,+}(z,h)\Big]_{j=0}^1= E_1(z,h) \varphi(r) E_{0,+}(z,h),
\end{equation}
\begin{equation}\label{f5}
\Big[E_{j,-}(z,h)\Big]_{j=0}^1= E_{0,-}(z,h) \varphi(r) E_1(z,h),
\end{equation}
and
\begin{equation}\label{f6}
\Big[E_{j,-+}(z,h)\Big]_{j=0}^1= E_{1,-}(z,h) \varphi(r)
E_{0,+}(z,h).
\end{equation}
In particular, if \ref{h1} is satisfied then
\begin{equation}\label{f7}
\Big[E_{j,-+}\big(k,r;z,h\big)\Big]_{j=0}^1\in S^\delta\Big(\mathbb
T^*_k\times {\mathbb R}^n_{r}; {\mathcal M}_N(\mathbb
C^N)\Big).
\end{equation}
\end{lemma}
\begin{proof}
Identities \eqref{f4}-\eqref{f6} follow from the first resolvent equation
\begin{align}
\Big[{\mathcal E}_j(z,h)\Big]_{j=0}^1=
{\mathcal E}_1(z,h)\big[{\mathcal P}_0(z,h)-{\mathcal P}_1(z,h)\big]
{\mathcal E}_0(z,h)\nonumber\\
\hskip2.6cm=
-{\mathcal E}_0(z,h)\big[{\mathcal P}_1(z,h)-{\mathcal P}_0(z,h)\big]
{\mathcal E}_1(z,h)\nonumber
\end{align}
and the fact that $\Big[{\mathcal
P}_j(z,h)\Big]_{j=0}^1=\begin{pmatrix} \varphi(r) & 0\\ 0 &0
\end{pmatrix}.$

Formula \eqref{f7} is a simple consequence of \eqref{f6} and
standard $h$-pseudodifferential calculus.
\end{proof}
\begin{lemma}\label{trace}
Assume \ref{h1} with \fbox{$\delta>n$}, the operator
\begin{align}\label{f8}
\varphi(r) E_{0,+}(z,h) : L^2(\mathbb T^*;\mathbb C^N) \rightarrow
L^2({\mathbb R}^n),
\end{align}
and
\begin{align}\label{f9}
E_{0,-}(z,h)\varphi(r): L^2({\mathbb R}^n)\rightarrow L^2(\mathbb
T^*;\mathbb C^N),
\end{align}
are of trace class.
\end{lemma}
\begin{proof} Since $\Big(E_{0,-}(z,h)\varphi(r)\Big)^*=\varphi(r)E_{0,+}(\overline{z},h)$
it suffice to prove \eqref{f8}. Without any loss of generality, we
may assume that $N=1$.

Consider the operator $A=\big({\rm Id}-h^2\Delta_{\mathbb
T^*}\big)^{-\frac{\delta}{2}}$ on $L^2(\mathbb T^*;\mathbb C).$
Set $B=\varphi(r)E_{0,+}(z,h)$, $C=B^*B$ and $D=A^{-1}CA^{-1}.$

Since $\varphi\in S^\delta(\mathbb R^{2n})$
and $E_{0,+}(k,r;z,h)\in S^0$, a standard result of $h$-pseudodifferential calculus
shows that $D\in S^0(\mathbb T^*_k\times{\mathbb R}^{n}_r).$
Therefore, $D$ extends to a bounded operator from $L^2(\mathbb
T^*;\mathbb C)$ into $L^2(\mathbb T^*;\mathbb C),$ see
\cite[Chapter 13]{DiSj99_01}. Combining this with the fact that
$C$ is positif, we get:
$$
0\le C=A D A
\le \Vert D\Vert \, A^2,
$$
which implies
$$
0\le C^{\frac{1}{2}}\le \sqrt{\Vert D\Vert}\, A.
$$
Since $\delta>n$ then $A:L^2(\mathbb T^*;\mathbb C) \rightarrow L^2(\mathbb T^*;\mathbb C)$
is of trace class and the lemma follows from the above inequality.
\end{proof}

\begin{remark}
Notice that if $P\in S^\delta({\mathbb T}^*\times\mathbb R^n; {\mathcal M}_N(\mathbb C))$
with $\delta>n$, then the operator $P^w(k,hD_k)$ is a trace class, see \cite{DiSj99_01}.
\end{remark}

\begin{proposition}\label{optr}
Assume \ref{h1} with \fbox{$\delta>n$}. For $z\in \Omega$ such that $\Im(z)\not=0,$ the operator
$$
\Big[
E_{j,+}(z,h)E_{j,-+}(z,h)^{-1} E_{j,-}(z,h)\Big]_{j=0}^1
$$
is of trace class
from $L^2(\mathbb R^n)$ to $L^2(\mathbb R^n)$ and
\begin{equation}\label{f10}
 {\rm tr}\Big(\Big[
E_{j,+}(z,h)E_{j,-+}(z,h)^{-1} E_{j,-}(z,h)\Big]_{j=0}^1\Big)=
{\rm tr}\Big(\Big[
E_{j,-+}(z,h)^{-1}\partial_z E_{j,-+}(z,h)
\Big]_{j=0}^1\Big).
\end{equation}
Here the operator in the right member of \eqref{f10} is defined
on $L^2(\mathbb T^*;\mathbb C^N).$
\end{proposition}
\begin{proof}
Let $z\in \Omega$ such that $\Im(z)\not=0,$ we have the following identity:
\begin{align}\label{ftr}
& \Big[E_{j,+}(z,h)E_{j,-+}(z,h)^{-1}
E_{j,-}(z,h)\Big]_{j=0}^1= \\
& \Big[\big([E_{j,+}(z,h)]_{j=0}^1\big)
E_{1,-+}(z,h)^{-1} E_{1,-}(z,h)
\Big]+ \nonumber \\
& \Big[E_{0,+}(z,h) E_{0,-+}(z,h)^{-1}
\big([E_{j,-}(z,h)]_{j=0}^1\big)
\Big]-  \nonumber\\
& \Big[E_{0,+}(z,h) E_{1,-+}(z,h)^{-1}
\big([E_{j,-+}(z,h)]_{j=0}^1\big)
 E_{0,-+}(z,h)^{-1} E_{1,-}(z,h)
\Big]. \nonumber
\end{align}
According to Lemmas \ref{formula} and \ref{trace}, all the term of
the right member in the last equality are of trace class. Using the
cyclicity of the trace and identity \eqref{f3}, we obtain the
proposition.
\end{proof}

Using again the cyclicity of the trace in \eqref{ftr} and the identity \eqref{f3} we obtain
\begin{equation}\label{cycltr}
\text{tr}\Big(\Big[
E_{j,-+}(z,h)^{-1}\partial_z E_{j,-+}(z,h)
\Big]_{j=0}^1\Big)=
\text{tr}\Big(\Big[
\partial_z E_{j,-+}(z,h) E_{j,-+}(z,h)^{-1}
\Big]_{j=0}^1\Big).
\end{equation}

The main result in this subsection is
\begin{proposition}\label{HeSjFormula}
Assume \ref{h1} with \fbox{$\delta>n$}. Let $\psi\in {\mathcal C}_{0}^{\infty}(\mathbb R)$
and let $\widetilde \psi$ be an almost analytic extension
of $\psi.$
Then the operator $[\psi(P(h))-
\psi(P_{0})]$ is of trace class as an operator
from $L^2(\mathbb R^n)$ to $L^2(\mathbb R^n)$ and
\begin{align}\label{f11}
& {\rm tr}[\psi(P(h))-\psi(P_{0})]={\rm tr}[\psi({\mathbb
P}_1(h))-\psi({\mathbb P}_{0}(h))]=
\\
& \hskip160pt -\frac{1}{\pi}\int_{\mathbb C}\overline{\partial}
\widetilde \psi(z) {\rm tr}\Big(\big[E_{j,-+}(z)^{-1}
\partial_{z}E_{j,-+}\big]_{j=0}^1\Big)\, L(dz).\nonumber
\end{align}
Here $\displaystyle\overline{\partial}=
\frac{\partial}{\partial\overline{z}}$ and $L(dz)=dxdy$ denotes the
Lebesgue measure on $\mathbb C.$
\end{proposition}

Recall that $\widetilde\psi\in {\mathcal C}_0^\infty(\mathbb{C})$ is an almost analytic extension
of $\psi,$ i.e. $\widetilde\psi_{|\mathbb R}=\psi$ and
$\overline{\partial}\widetilde\psi={\mathcal O}(|\Im(z)|^N)$ for all $N\in \mathbb N.$
We refer to \cite{DiSj99_01} for the existence of $\widetilde \psi.$

\begin{proof}
By Helffer-Sj\"ostrand formula (see \cite{HeSj89_01}), we have
\begin{equation}\nonumber
\psi({\mathbb P}_1(h))-\psi({\mathbb P}_{0}(h))=-\frac{1}{\pi}
\int_{\mathbb C} \overline{\partial} \widetilde \psi(z)
 \Big[(z-{\mathbb P}_1(h))^{-1}-(z-{\mathbb P}_0(h))^{-1}\Big]\,L(dz).
\end{equation}
Combining this with \eqref{f1}, we obtain
\begin{align}\label{f12}
&\psi({\mathbb P}_1(h))-\psi({\mathbb P}_{0}(h))=\frac{1}{\pi} \int_{\mathbb C}
\overline{\partial} \widetilde \psi(z)
 \big[E_j(z,h)]_{j=0}^1\,L(dz)\\
&\hskip100pt-\frac{1}{\pi} \int_{\mathbb C} \overline\partial
\widetilde \psi(z)
\Big[E_{j,+}(z,h)E_{j,-+}(z,h)^{-1}E_{j,-}(z,h)\Big]_{j=0}^1\,L(dz).\nonumber
\end{align}
Since $E_j(z,h),\,j=0,1$ is holomorphic in a neighborhood
of $\text{supp} (\widetilde \psi),$ the first term in the right
member of \eqref{f12} vanishes. Consequently,
\begin{equation}\nonumber
\psi({\mathbb P}_1(h))-\psi({\mathbb P}_{0}(h))= -\frac{1}{\pi}
\int_{\mathbb C} \overline\partial \widetilde \psi(z)
\Big[E_{j,+}(z,h)E_{j,-+}(z,h)^{-1}E_{j,-}(z,h)
\Big]_{j=0}^1\,L(dz).
\end{equation}
Using Proposition \ref{optr}, we conclude that $[\psi({\mathbb
P}_{1}(h))-\psi({\mathbb P}_{0}(h))]$ is of trace class and applying
\eqref{f10}, we obtain the second equality of \eqref{f11}. The first
equality follows from the fact that ${\mathbb P}_1(h)$ (resp.
${\mathbb P}_0(h)$)
 is unitary equivalent to $P(h)$ (resp.
$P_0$).
\end{proof}

Now, we recall  a representation of the derivative of the spectral
shift function in term of the effective Hamiltonian
$E_{j,-+}(z,h)$ (see \cite[Lemma 1]{Di05_01}).

Let $I\subset\mathbb R$ be some bounded interval and $\Omega$ be the complex neighborhood
of $I$ given by the proposition \ref{Grushinproblem}.
Put $\Omega_{\pm}=\Omega\cap\{z\in \mathbb C;\ \pm \Im(z)>0\}.$ \\
We introduce the functions $e_{\pm}(z)={\rm tr}\Big(\big[E_{j,-+}(z,h)^{-1}
\partial_{z}E_{j,-+}(z,h)\big]_{j=0}^1\Big),$ for $z\in \Omega_{\pm}.$\\
Since $E_{j,-+}(z,h)$ is holomorphic on $z$, we deduce from \eqref{adjcc} that
$$
\partial_{z}E_{j,-+}(z,h)^{*}=\partial_{z}E_{j,-+}(\overline{z},h).
$$
Using the fact that $\overline{{\rm tr}(A)}={\rm tr}(A^*)$ with \eqref{adjcc},
\eqref{cycltr} and the above equality we obtain
\begin{equation}\label{f13}
 \overline{e_{+}(z)}=e_{-}(\overline{z}),\quad \text{for all}\,\, z\in\Omega_{+}.
\end{equation}

\begin{lemma}\label{representationSSF} \cite[Lemma 1]{Di05_01}
Assume \ref{h1} with \fbox{$\delta>n$}. In ${\mathcal D}'(I),$ we
have
\begin{equation}\label{repSSF}
\zeta_h'(\mu)= \lim_{\epsilon\searrow 0} \frac{1}{2\pi i}
\Big[e_{+}(\mu+i\epsilon)-\overline{e_{+}(\mu+i\epsilon)}\Big].
\end{equation}
\end{lemma}

For the reader convenience we give the proof of the lemma.
\begin{proof}
Let $\psi\in {\mathcal C}_0^\infty(I).$ In the previous proposition, we proved that
\begin{equation}\nonumber
-\langle \zeta_h'(\cdot),\psi\rangle= {\rm
tr}[\psi(P(h))-\psi(P_{0})]= -\frac{1}{\pi}\int_{\mathbb
C}\overline{\partial} \widetilde \psi(z) {\rm
tr}\Big(\big[E_{j,-+}(z)^{-1}
\partial_{z}E_{j,-+}\big]_{j=0}^1\Big)\, L(dz).
\end{equation}

Since $e_{\pm}(z)={\mathcal O}\big(h^{-n}\vert\Im(z)\vert^{-2}\big)$
and $\overline{\partial} \widetilde \psi(z)={\mathcal
O}\big(\vert\Im(z)\vert^{2}\big),$ the integral in the identity
above converge. Thus the r.h.s. of the previous equality can be
written as
\begin{equation}\nonumber
-\langle \zeta_h'(\cdot),\psi\rangle=\lim_{\epsilon\searrow
0}-\frac{1}{\pi} \Big[\int_{\Im(z) >0}\overline{\partial} \widetilde
\psi(z) e_{+}(z+i\epsilon)\, L(dz)+ \int_{\Im(z)
<0}\overline{\partial} \widetilde \psi(z) e_{-}(z-i\epsilon)\,
L(dz)\Big].
\end{equation}
Since $e_{\pm}(z\pm i\epsilon)$ is holomorphic in $\Omega_{\pm},$ Green's formula then gives
\begin{equation}\label{f14}
-\langle \zeta_h'(\cdot),\psi\rangle=\lim_{\epsilon\searrow
0}\frac{i}{2\pi} \int_{\mathbb R}
\psi(\mu)\big[e_{+}(\mu+i\epsilon)-e_{-}(\mu-i\epsilon)\big]d\mu.
\end{equation}
Thus Lemma \ref{representationSSF} follows from \eqref{f13} and
\eqref{f14}.
\end{proof}

In the following we will use the result concerning the expression of the effective hamiltonian
$E_{j,-+}(z,h), j=0,1$ given in \cite{Di93_01} (see also \cite{Di05_01}).
In fact, under the assumptions \ref{h1}-\ref{h2},
the two leading terms of the symbol $E_{1,-+}(k,r;z,h)$ are computed in
\cite[section 4, formulas (4.5)-(4.7)]{Di93_01}, it was shown that:
\begin{equation}\label{cal}
E_{1,-+}(k,r;z,h)=z-\big(\lambda(k)+\varphi(r)+hK_1(k,r)+h^2K_2(k,r;z,h)\big),
\end{equation}
where $K_1\in S^{\delta+1}(\mathbb T^*\times {\mathbb R}^n)$ and
$K_2(\cdot;z,h)\in S^{\delta+2}({\mathbb T}^*\times {\mathbb R}^n),$
holomorphic with respect to $z$ in $\Omega$. Note that
\begin{equation}\label{cal1}
 E_{0,-+}(k,r;z)=z-\lambda(k),\quad k\in {\mathbb T}^*,\,z\in \Omega.
\end{equation}

From now on, we consider the $h$-pseudodifferential operator $H^w(k,-hD_k;h)$ with the following symbol:
\begin{equation}\label{opzind}
H(k,r;h)=\lambda(k)+\varphi(r)+hK_1(k,r).
\end{equation}
Remark that this operator is $z$-independent.

\begin{corollary}\label{repre2}
Under Assumptions \ref{h1} with \fbox{$\delta>n$} and \ref{h2},
there exists $\displaystyle G(k,r;z,h)\sim\sum_{j=0}^\infty g_j(k,r;z)h^j$
in $S^{\delta+2}({\mathbb T}^*\times {\mathbb R}^n)$ such that for $\mu\in I$ and $h$ small enough, we have:
\begin{equation}\label{repSSF1}
\zeta_h'(\mu) =
\lim_{\epsilon\searrow 0} \frac{1}{2\pi i}
\left[{\rm tr}\Big((z-B_\mu)^{-1}-
\big(z-\lambda(k)\big)^{-1}\Big)\right]^{z=\mu+i\epsilon}_{z=\mu-i\epsilon},
\end{equation}
where $B_\mu:= H^w(k,-hD_k;h)+h^2G^w(k,-hD_k,\mu;h)$. Here $H(k,r;h)$ is given by \eqref{opzind}.
\end{corollary}

\begin{proof} Identity \eqref{cal} gives $\partial_z E_{1,-+}(k,r;z,h)=1+h^2\partial_z K_2(k,r;z,h)$
and since $\partial_z K_2\in S^{\delta+2}({\mathbb T}^*\times
{\mathbb R}^n)\in S^0({\mathbb T}^*\times {\mathbb R}^n)$ it follows
from the Calderon-Vaillancourt's theorem and the Beal's
characterization (see \cite{DiSj99_01}) that the corresponding
operator $\partial_z E_{1,-+}(z,h)$ is invertible for $h$ small
enough and his inverse is given by
$$
\big[\partial_z E_{1,-+}(z)\big]^{-1}= I+h^2{\mathcal R}(z),
$$
where ${\mathcal R}(z)$ is an $h$-pseudodifferential operator with
symbol satisfying the same properties as $K_2$. Combining this with
\eqref{cal} and using the composition formula of
$h$-pseudodifferential operators we see that there exists
$G\sim\sum_{j=0}^\infty g_j(k,r;z)h^j$ in $S^{\delta+2}({\mathbb
T}^*\times {\mathbb R}^n)$ such that
$$
\big(\partial_z E_{1,-+}(z,h)\big)^{-1} E_{1,-+}(z,h)=z-H^w(k,-hD_k;h)+h^2 G^w(k,-hD_k;z,h),
$$
which together with \eqref{cal1}, Lemma \ref{representationSSF} and the holomorphy of $G$ on $z$ give the corollary.

\end{proof}

\Subsection{Proof of the weak asymptotic expansion of $\xi'_h(\cdot)$}\label{pw-SSFh}

Let $I=]a,b[\in\mathbb R$ and $f\in {\mathcal C}_0^\infty(I).$
The proof of Theorem \ref{w-SSFh} is a simple consequence
of Proposition \ref{HeSjFormula}
(with $\psi=f$)  and symbolic calculus.
Here, we only give an outline of the proof.
For the details, we refer to \cite{Di93_01}.
Fix  $\epsilon$ in $\rbrack 0,\frac{1}{2}\lbrack$.
The integral \eqref{f11} over $\{z\in{\mathbb C};\ \vert \Im z\vert\le h^\epsilon\}$
is ${\mathcal O}(h^\infty)$,
since
$\overline\partial \widetilde f(z)= {\mathcal O}(\vert \Im z\vert^\infty)$ and
$$
\Big\Vert \Big[E_{j,-+}(z)^{-1} \partial_z E_{j,-+}(z)\Big]_{j=0}^1
\Big\Vert_{{\rm tr}}={\mathcal O}\big(\vert \Im z\vert^{-1}\big).
$$

On the other hand,
$\Big[E_{j,-+}(z,h)^{-1} \partial_z E_{j,-+}(z,h)\Big]_{j=0}^1$
has an asymptotic expansion in powers of $h$ uniformly for
$z$ in $\{z\in {\rm supp}\widetilde f;\,$
$|\Im(z)|\ge h^\epsilon\}$
(see \cite{Di93_01}). Therefore, as in \cite[Theorem 13.28]{DiSj99_01}, we have
$$
-\frac{1}{\pi}\int_{\mathbb C}
\overline\partial \widetilde f(z) {\rm tr}
\Big(\big[E_{j,-+}(z)^{-1}
\partial_z E_{j,-+}(z)\big]_{j=0}^1\Big)\,L(dz)
\sim h^{-n}\sum_{j\geq 0}
a_j h^{j},\,\, (h\searrow 0)
$$
with
$$
a_0=(2\pi)^{-n}\sum_{p\geq 1}\int_{{\mathbb R}^n_x}\Big(\int_ {E^*}
\big[f\big(\lambda_p(k)+\varphi(x)\big)-
f\big(\lambda_p(k)\big)\big]\,dk\Big)\,dx.
$$
Note that, the sum in the last equality is finite, since
$\displaystyle\lim_{p\rightarrow +\infty}\lambda_p(k) =+\infty$ and
$\varphi$ is bounded.

The coefficient $a_j$ is a finite sum of term of the form $\int\int
c_l(x,k) f^{(l)}\big(b(x,k)\big)\, dxdk,$ where $c_l$ depends on
$\varphi$ and their derivatives and $b(x,k)\in\big\{\lambda_p(k),
\lambda_p(k)+\varphi(x)\big\},$ which complete the first part of the
theorem \ref{w-SSFh}. See \cite[Chapter 8, Identity
(8.16)]{DiSj99_01}.

If $\mu$ is a non-critical energy of $P_0$ for all $\mu\in
I.$  Then $d\big(\lambda_p(k)\big)\not=0$ and
$d\big(\lambda_p(k)+\varphi(x)\big)\not=0$ for all $k\in
F(\mu).$ We recall that $F(\mu)$ is the Fermi surface.
Therefore, $a_j(f)=-\langle \gamma_j(\cdot),f\rangle,$ for all $f\in
{\mathcal C}_0^\infty(I)$ and $\gamma_j(\mu)$ are smooth
functions of $\mu\in I,$ in particular,
$\gamma_0(\mu)=\frac{\rm d}{\text{d}\mu}\left[\int_{\mathbb
R^n_x}\rho\big(\mu\big)-\rho\big(\mu-\varphi(x)\big) \,
dx\right],$ which complete the proof of the theorem \ref{w-SSFh}.
\hfill{$\square$}

\Subsection{Limiting absorption Theorem}\label{abslimTh}

In this subsection we establish a limiting absorption principle for the operator $P(h)$,
see Theorem \ref{abslimit1} below.
We start by the following lemma, let $[a,b]\in \mathbb R$ and $\Omega$ given by Proposition \ref{Grushinproblem}.
\begin{lemma}\label{abslimit}
Assume that \ref{h1}, \ref{h2} and \ref{h3} are satisfied on $[a,b]$. Then, for
$l\in \mathbb{N}^*,$ there exists $h_0(l)>0$ small enough such that for all
$0<h<h_0(l)$:
\begin{equation}\label{estabslim}
\big\Vert <hD_k>^{-\alpha} E_{j,-+}(k,-hD_k,\mu\pm i0;h)^{-l}
<hD_k>^{-\alpha} \big\Vert={\mathcal O}(h^{-l}), \quad (j=0,1),
\end{equation}
for all $\alpha>l-\frac{1}{2}$ and uniformly for $\mu\in [a,b],$
where $\Vert\cdot\Vert$ denotes the operator norm when considered as
an operator from $L^2({\mathbb T}_k^*)$ into itself.
\end{lemma}

\begin{proof} Recall that $\Omega$ is a small complex neighborhood of $[a,b]$ 
such that the proposition \ref{Grushinproblem} holds.
For $s\in\Omega \cap\mathbb R$, we consider the
$h$-pseudodifferential operator $\widetilde
P_s(k,-hD_k;h):=H^w(k,-hD_k;h)+h^2K_2^w((k,-hD_k;s,h)$ with the
following symbol:
$$
\widetilde P_s(k,r;h)=H(k,r;h)+h^2K_2(k,r;s,h),
$$
where $H(k,r;h)$ given by \eqref{opzind} and $K_2(\cdot;z,h)\in
S^{\delta+2}({\mathbb T}^*\times {\mathbb R}^n),$ holomorphic with
respect to $z$ in $\Omega$. Put $A={\rm Op}_h^w\big(-\nabla
\lambda(k)\cdot r\big).$ Since $K_2\in S^{\delta+2}({\mathbb
T}^*\times {\mathbb R}^n)\subset S^{0}({\mathbb T}^*\times {\mathbb
R}^n)$, it follows from  the $h$-pseudodifferential calculus and the
Calderon-Vaillancourt's theorem that
$$
\big[\widetilde P_s,A\big]={\rm
Op}_h^w\Big(\big\vert\nabla_k\lambda(k)\big\vert^2-
\big(-r\nabla_r\varphi(-r)\big)\cdot\Delta\lambda(k)\Big)+{\mathcal
O}(h),
$$
in ${\mathcal L}(L^2({\mathbb T}^*))$, uniformly for $s\in \Omega\cap \mathbb R$.
Here $[\cdot,\cdot]$ denotes the commutator.

The assumption \ref{h3} and the G\"arding  inequality (see [11,
chapter 8]) imply that for $f\in {\mathcal C}^\infty_0(\Omega\cap\mathbb R)$
there exists $C>0$ such that
$$
f(\widetilde P_s)\big[\widetilde P_s,A\big]f(\widetilde
P_s)\geq C f(\widetilde P_s)^2\quad \text{for}\ h\
\text{small enough and uniformly for} \ s\in \Omega\cap \mathbb R.
$$
Now applying \cite[Theorem 1]{Ge08_01}, we get, for all $l\in\mathbb N^*,$
\begin{equation}\label{wo1}
\big\Vert <hD_k>^{-\alpha} \big(\mu\pm i0-\widetilde
P_s\big)^{-l} <hD_k>^{-\alpha}\big\Vert={\mathcal O}(h^{-l}),
\quad \text{for all}\ \,\alpha>l-\frac{1}{2},
\end{equation}
uniformly for $\mu\in [a,b]$ and $s\in \Omega\cap\mathbb R.$ Take
$s=\mu$ in \eqref{wo1} we obtain the estimation \eqref{estabslim} for $j=1$.
\end{proof}

\begin{theorem}[Limiting absorption Theorem]\label{abslimit1}
With the same assumptions as Lemma \ref{abslimit}. One has, for
$l\in\mathbb{N}^*,$ there exists $h_0(l)>0$ small enough such that for all
$0<h<h_0(l)$:
\begin{equation}\label{estabslimP}
\big\Vert <hx>^{-\alpha} \big(P(h)-\mu\pm i0\big)^{-l}
<hx>^{-\alpha} \big\Vert_{{\mathcal L}\big(L^2({\mathbb
R}_x^n)\big)}={\mathcal O}(h^{-l}),
\end{equation}
for all $\alpha>l-\frac{1}{2}$ and
uniformly for $\mu\in [a,b].$
\end{theorem}

\begin{proof}
Recall that the operator $P(h)$ acting on
$L^2(\mathbb R^n)$ with domain $H^2(\mathbb R^n)$ is unitary
equivalent to
$$
{\mathbb P}_1(h):=\big(D_y+hD_x\big)^2+V(y)+\varphi(x),
$$
acting on ${\mathbb L}^0$ with domain ${\mathbb L}^2$,
where $\displaystyle{\mathbb
L}^m:=\{u(x)T_{\Gamma}(x,y);\
\partial^{\alpha}_xu\in L^2(\mathbb R^n),\
\forall\alpha,\,|\alpha|\leq m\}$ for $m\in \mathbb N$
and $T_{\Gamma}$ is a distribution in ${\mathcal S}'(\mathbb R^{2n})$ defined by
$$
T_{\Gamma}(x,y)=\frac{1}{\text{vol}(E)h^{n}}
\sum_{\beta^*\in \Gamma^*}e^{i(x-hy)\frac{\beta^*}{h}}.
$$
Here $E$ is a fundamental domain of $\Gamma.$ Then we will prove \eqref{estabslimP}
for ${\mathbb P}_1(h).$

It follows from identity \eqref{f1} that:

\begin{align}
& \langle hy\rangle^{-\alpha}\big({\mathbb P}_1(h)-z\big)^{-1}\langle hy\rangle^{-\alpha}=
\langle hy\rangle^{-\alpha}E_1(z,h)\langle hy\rangle^{-\alpha}-\langle hy\rangle^{-\alpha}E_{1,+}(z,h)
\langle hD_k\rangle^{\alpha}\cdot \nonumber \\
& \hskip140pt\cdot\langle hD_k\rangle^{-\alpha}E_{1,-+}(z,h)^{-1}\langle hD_k\rangle^{-\alpha}
\cdot \langle hD_k\rangle^{\alpha} E_{1,-}(z,h)\langle hy\rangle^{-\alpha}. \nonumber
\end{align}
Since $E_1(z,h)$ is holomorphic then the first term of the r.h.s is bounded.
On the other hand, as in Proposition
\ref{Grushinproblem}, we prove that

\begin{align}
& \begin{pmatrix}
\langle hy\rangle^{-\alpha}\big({\mathbb P}_1(h) -z\big)\langle hy\rangle^{\alpha}
  & \langle hy\rangle^{-\alpha}R_+^*\langle hD_k\rangle^{\alpha} \\
\langle hD_k\rangle^{-\alpha} R_+\langle hy\rangle^{\alpha} & 0
\end{pmatrix}=\hskip100pt\nonumber \\
& \hskip150pt \begin{pmatrix}
\langle hy\rangle^{-\alpha}  & 0 \\
 0 & \langle hD_k\rangle^{-\alpha}
\end{pmatrix}
\begin{pmatrix}
{\mathbb P}_1(h) -z  & R_+^* \\
 R_+ & 0
\end{pmatrix}
\begin{pmatrix}
\langle hy\rangle^{\alpha}  & 0 \\
 0 & \langle hD_k\rangle^{\alpha}
\end{pmatrix}\nonumber
\end{align}
\noindent
is well-defined as a bounded operator from ${\mathbb L}^{2}\times L^2({\mathbb T}^*;\mathbb C^N)$
to ${\mathbb L}^{0}\times L^2({\mathbb T}^*;\mathbb C^N)$
and is bijective with bounded two-sided inverse given by:

\begin{align}
&\begin{pmatrix}
 \langle hy\rangle^{-\alpha}E_1(z,h)\langle hy\rangle^{\alpha}
  & \langle hy\rangle^{-\alpha}E_{1,+}(z,h)\langle hD_k\rangle^{\alpha} \\
\langle hD_k\rangle^{-\alpha}E_{1,-}(z,h) \langle hy\rangle^{\alpha} &
\langle hD_k\rangle^{-\alpha}E_{1,-+}(z,h) \langle hD_k\rangle^{\alpha}
\end{pmatrix} =\hskip70pt\nonumber\\
& \hskip100pt  \begin{pmatrix}
\langle hy\rangle^{-\alpha}  & 0 \\
 0 & \langle hD_k\rangle^{-\alpha}
\end{pmatrix}
\begin{pmatrix}
E_1(z,h)  &  E_{1,+}(z,h)\\
E_{1,-}(z,h) & E_{1,-+}(z,h)
\end{pmatrix}
\begin{pmatrix}
\langle hy\rangle^{\alpha}  & 0 \\
 0 & \langle hD_k\rangle^{\alpha}
\end{pmatrix}.\nonumber
\end{align}
Then $\langle hy\rangle^{-\alpha}E_{1,+}(z,h)\langle hD_k\rangle^{\alpha}$
and $\langle hD_k\rangle^{-\alpha}E_{1,-}(z,h) \langle hy\rangle^{\alpha}$
are well-defined and bounded. Therefore combining this with Lemma \ref{abslimit}
we obtain Theorem \ref{abslimit1} for $l=1$.
With the same arguments we obtain the result for $l\geq 2$.
\end{proof}

\begin{remark}
A simple consequence of Theorem \ref{abslimit1} is that $P(h)$ has no
embedded eigenvalues in $[a,b]$.
\end{remark}

\Subsection{Proof of the pointwise asymptotic expansion of  $\xi'_h(\cdot)$}\label{p-dSSFh}
Recall that $\Omega$ is a small complex neighborhood of $[a,b]$
such that the proposition \ref{Grushinproblem} holds.
For $s\in\Omega \cap\mathbb R$, we consider the $h$-pseudodifferential operator
$B_s(k,-hD_k;h):=H^w(k,-hD_k;h)+h^2G^w(k,-hD_k;s,h)$
with the following symbol:
$$
b_s(k,r;h)=H(k,r;h)+h^2G(k,r;s,h),
$$
where $H(k,r;h)$ given by \eqref{opzind} and $G(\cdot;z,h)\in
S^{\delta+2}({\mathbb T}^*\times {\mathbb R}^n),$ holomorphic with
respect to $z$ in $\Omega$ (see Corollary \ref{repre2}).

Clearly, the principal symbol of $B_s$ (i.e.
$b_s^0=\lambda(k)+\varphi(r)$) is independent on $s$. Moreover, the
assumption \ref{h3} implies that $[a,b]$ is a non-trapping region
for the classical hamiltonian $b_s$. Now, by constructing a
long-time parametrix for the time-dependent equation
$$
\big(hD_t-B_s\big)U(t)=0,\quad  U(0)=I,
$$
we can apply the Robert-Tamura method \cite{RoTa84_01, RoTa87_01, RoTa88_01}
(see also \cite[Remark 6.1]{Ro94_01}) to prove that
$$
\Big[{\rm tr}
\Big((z-B_s)^{-1}-(z-\lambda(k))^{-1}\Big)\Big]^{z=\mu+i0}_{z=\mu-i0}\,,
$$
has a complete asymptotic expansion in powers of $h$ uniformly for
$\mu\in [a,b]$ and $s\in \Omega\cap\mathbb R$. Remembering
\eqref{repSSF1} and take $s=\mu\in [a,b]\subset \Omega$ we obtain
\eqref{DimZerasym1}.

\end{document}